\documentclass{amsart}
\usepackage{amssymb}
\usepackage{amsthm}
\usepackage{indentfirst}
\usepackage{amsmath}
\usepackage{amscd}
\usepackage{xypic}

\usepackage{CJK}

\numberwithin{equation}{section}

\newtheorem{Theorem}{Theorem}[section]
\newtheorem{Definition}[Theorem]{Definition}
\newtheorem{Proposition}[Theorem]{Proposition}

\newtheorem{Assumption-Notation}[Theorem]{Assumption-Notation}

\newtheorem{Remark}[Theorem]{Remark}

\newtheorem{Claim}[Theorem]{Claim}

\newtheorem{Example}[Theorem]{Example}
\newtheorem{Conjecture}[Theorem]{Conjecture}

\def\dim{\operatorname{dim}}

\begin{document}

\title[$C_{n,n-1}$ in positive characteristics]{The subadditivity of the Kodaira Dimension for Fibrations of Relative Dimension One in Positive Characteristics}

\author{Yifei Chen}

\address{Y. Chen: Hua Loo-Keng Key Laboratory of Mathematics, Academy of Mathematics and Systems Science,
         Chinese Academy of Sciences,
         No. 55 Zhonguancun East Road,
         Haidian District, Beijing, 100190,
         P.R.China}
\email{yifeichen@amss.ac.cn}

\author{Lei Zhang}
\address{L. Zhang: College of Mathematics and Information Sciences, Shaanxi Normal University, Xi'an 710062, P.R.China}
\email{lzhpkutju@gmail.com}

\maketitle

\begin{abstract} Let $f:X\rightarrow Z$ be a separable fibration of relative dimension 1 between smooth projective varieties over an algebraically closed field $k$ of positive characteristic.  We prove  the subadditivity of Kodaira dimension $\kappa(X)\geq\kappa(Z)+\kappa(F)$, where $F$  is the  generic geometric fiber of $f$, and $\kappa(F)$ is the Kodaira dimension of the normalization of $F$. Moreover, if $\dim X=2$ and $\dim Z=1$, we have a stronger inequality $\kappa(X)\geq \kappa(Z)+\kappa_1(F)$ where $\kappa_1(F)=\kappa(F,\omega^o_F)$ is the Kodaira dimension of the dualizing sheaf $\omega_F^o$.
\end{abstract}

\section{Introduction}

The following conjecture due to Iitaka is a central problem of the classification theory of algebraic varieties over $\mathbb{C}$, the field of complex numbers:

\begin{Conjecture}[$C_{n,m}$] Let $f:X\rightarrow Z$ be a surjective morphism of proper, smooth varieties over $\mathbb{C}$, where $n=\dim X$ and $m=\dim Z$. Assuming the generic geometric fibre $F$ of $f$ is connected, we have the following inequality for the Kodaira dimension:
  $$\kappa(X)\geq \kappa(F)+\kappa(Z).$$
\end{Conjecture}

This conjecture is proved in the following cases:

\begin{itemize}
\item[(1)]{$\dim F = 1,2$ by Viehweg (\cite{Vie1},\cite{Vie});}
\item[(2)]{$Z$ is of general type by Kawamata (\cite{Ka0});}
\item[(3)]{$\dim Z=1$ by Kawamata (\cite{Ka1});}
\item[(4)]{$F$ has a good minimal model by Kawamata (\cite{Ka2});}
\item[(5)]{$F$ is of general type by Koll\'{a}r (\cite{Ko});}
\item[(6)]{$Z$ is of maximal Albanese dimension by J. A. Chen and Hacon (\cite{CH});}
\item[(7)]{$F$ is of maximal Albanese dimension by Fujino (\cite{Fu});}
\item[(8)]$n\leq 6$  by Birkar (\cite{bi}).
\end{itemize}

Since the existence of good minimal models of $F$ is proved for certain cases (see \cite{bchm},\cite{Lai}), (1,5,7) are special cases of (4).

In this paper, we consider subadditivity of Kodaira dimension of a fibration of relatively dimension 1 over an algebraically closed field $k$ of positive characteristic.  We call a projective surjective morphism $f:X\rightarrow Z$ between two varieties a \emph{fibration}, if $f_*\mathcal{O}_X=\mathcal{O}_Z$. We say a fibration $f:X\rightarrow Z$ is \emph{separable} if the field extension $f^*:k(Z)\rightarrow k(X)$ between the rational function fields is separable and $k(Z)$ is algebraically closed in $k(X)$. Then general fiber of a separable fibration is geometrically integral (\cite{ba01} Theorem 7.1).

Over a field of positive characteristic, the generic geometric fiber of $f$ is possibly not smooth. So a proper definition of Kodaira dimension for singular varieties is needed. We have two attempts (Definition \ref{D:Kod} and \ref{D:Kod1}).

 The main results of the paper are:

\begin{Theorem}\label{T:Cnn-1}
Let $f: X\rightarrow Z$ be a separable fibration of relative dimension 1 between smooth projective varieties over an algebraically closed field $k$ of positive characteristic. Then
\begin{equation}\label{E:Cnn-1}
\kappa(X) \geq \kappa(Z) + \kappa(F),
\end{equation}
where $\kappa(F)=\kappa(\tilde{F},\omega_{\tilde{F}})$ (see Definition \ref{D:Kod}), $F$ the generic geometric fiber of $f$, $\tilde{F}$ the normalization of $F$, and $\omega_{\tilde{F}}$  the canonical line bundle of $\tilde{F}$.
\end{Theorem}

\begin{Theorem} \label{T:C21}
Let $f:S\rightarrow C$ be a fibration between a smooth projective surface $S$ and a smooth projective curve $C$ over an algebraically closed field $k$ of positive characteristic. Then
\begin{equation}\label{E:C21}
\kappa(S)\geq \kappa_1(F)+\kappa(C),
\end{equation}
where $F$ is the generic geometric fiber of $f$,  $\kappa_1(F):=\kappa(F,\omega_F^o)$  and $\omega_F^o$ is the dualizing sheaf of $F$ (see Definition \ref{D:Kod1}).
\end{Theorem}

Theorem \ref{T:C21} is essentially known to experts, as Prof. F. Catanese pointed out. For reader's convenience, we shall give a proof in Section \ref{S:C21}. We mainly apply Bombieri-Mumford's classification of surfaces in positive characteristics. We have a feeling that the inequality (\ref{E:C21}) reflects the geometry better than inequlity (\ref{E:Cnn-1}) (see Remark \ref{R:general_surf}). One difficulty to generalize the inequality (\ref{E:C21}) to higher dimensional cases is the weak positivity of $f_*\omega_{X/Z}^l$ for $l\gg0$, which plays an important role in the proof of the known cases of Conjecture $C_{n,m}$ over $\mathbb{C}$, but the positivity  fails in positive characteristics (see Remark \ref{R:K_not_nef}). The semi-positivity of $f_*\omega_{X/Z}^l$ for $l\gg0$ holds for certain fibrations  due to Patakfalvi (\cite{Pa}).

Let's  explain briefly the idea of  proof of Theorem \ref{T:Cnn-1}. We mainly follow Viehweg's approach (\cite{Vie1}) and study the behavior of the relative canonical sheaves under base changes or alterations. Starting from the fibration $f: X \rightarrow Z$, after a base change $\pi:Z' \rightarrow Z$, where $\pi$ is an alteration and possibly contain a purely inseparable extension,  we obtain an alteration $X'\rightarrow X$ such that, $X'$ and $Z'$ are smooth and the natural fibration $f': X' \rightarrow Z'$ factors through a stable (or 1-pointed stable if $g=1$) fibration  $f'': X''\rightarrow Z'$ with $X''$ having mild singularities (see the commutative diagram (\ref{Diagram}) in Section \ref{S:Stable_Red}). The stableness of the fibration implies that $\kappa(X'',\omega_{X''/Z'})\geq \kappa(\tilde{F})$ (Theorem \ref{ggt2}, \ref{g1}). We can attain the goal by showing $\kappa(X'',\omega_{X''/Z'})\leq\kappa(X',\omega_{X'/Z'})\leq \kappa(X,\omega_{X/Z})$. The first inequality is  by Proposition \ref{pf} since $X''$ has mild singularities. The second one is from carefully comparing $\omega_{X'/Z'}$ and the pull-back of $\omega_{X/Z}$( see Theorem \ref{compds}).

One difficulty to carry Viehweg's proof of $C_{n,n-1}$  into positive characteristics is the lack of  resolution of singularities in positive characteristics. Without resolution, there is no stable reduction by a flat base change as in \cite{Vie1} Theorem 5.1: where Viehweg shows that

1) there is a flat base change $Z'\rightarrow Z$ such that $X\times_ZZ'$ has mild singularities after replacing $f:X\rightarrow Z$ by a birational smooth model

2) the fibration $X\times_ZZ'\rightarrow Z'$ is birational to a stable fibration $f_s:X_s\rightarrow Z'$.

The advantage of flat base change is that the relative canonical sheaves (or relative dualizing sheaves) are compatible with flat base changes. In our proof, we construct a similar commutative diagram (Diagram (\ref{Diagram}) in Section \ref{S:Stable_Red}) by using de Jong's alteration, while the base change $Z'\rightarrow Z$ is not flat.

\smallskip

{\textbf{Acknowledgements.} Part of the work is done during the first author visiting DPMMS at University of Cambridge. The first author should thank Caucher Birkar's invitation,  DPMMS at Cambridge for stimulating research environment, and the financial support from AMSS, Chinese Academy of Sciences. The first author thanks many useful comments and discussions from people in the seminar organized by Birkar. The authors also thank  Yi Gu, Lingguang Li,  Jinsong Xu and Ze Xu for many useful discussions. The authors are grateful to Weizhe Zheng for simplifying the arguments of Theorem \ref{compds} in the early version of the paper. The first author is supported by NSFC (No. 11201454 and No. 11231003). The second author is supported by NSFC (No. 11226075).

\section{Prepliminaries and Notations}

\textbf{Notations and assumptions:}
\begin{itemize}
\item
$D(X)$ (resp. $D^+(X), D^-(X), D^b(X)$): the derived (low bounded, up bounded, bounded) category of quasi-coherent sheaves on a variety $X$;
\item
$\simeq$: quasi-isomorphism between two objects in the derived category;
\item
$Lf^*, Rf_*$: the derived functors of $f^*, f_*$ respectively for a morphism $f: X \rightarrow Y$;
\item
$g(C), p_a(C)$: the geometric genus and the arithmetic genus of the curve $C$.
\end{itemize}

For basic knowledge of derived categories, we refer to \cite{Hu}. All varieties are over an algebraically closed field $k$ of positive characteristic. On a variety $X$, we always identify a sheaf $\mathcal{F}$ with an object in $D(X)$. Let  $L_1,L_2$ be two line bundles on $X$. We denote $L_1\geq L_2$, if  $L_1\otimes L_2^{-1}$ has non-zero global sections.

\subsection{Duality theory} \label{dth}
We list some results on Grothendieck  duality theory, part of which appeared  in \cite{Vie1} \S 6. For details we refer to \cite{Ha} Chap. III Sec. 8, 10, Chap. V Sec. 9 and Chap. VII Sec. 4.

Let $f: X \rightarrow Y$ be a projective morphism between two projective varieties of relative dimension $r = \dim X - \dim Y$. There exists a functor $f^!: D^+(Y) \rightarrow D^+(X)$ such that for $F \in D^-(X)$ and $G \in D^+(Y)$,
$$Rf_*R\mathcal{H}om_X(F, f^!G) \simeq R\mathcal{H}om_Y(Rf_*F, G).$$
Recall that
\begin{itemize}
\item[(a)]{If $g \circ f: X \rightarrow Y \rightarrow Z$ is a composite of projective morphisms, then $(g\circ f)^! \simeq f^! \circ g^!$;}
\item[(b)]{For a flat base change $u: Y' \rightarrow Y$, there is an isomorphism $v^*f^! = g^!u^*$ where $v$ and $g$ are the two projections of $X\times_Y Y'$;}
\item[(c)]{If $G \in D^b(Y)$ is an object of finite Tor-dimension (\cite{Ha}; p. 97), there is a functorial isomorphism  $ f^!F \mathop{\otimes}^L Lf^*G \simeq f^!(F \otimes^L G)$ for $F \in D^+(Y)$ (\cite{Ha} p. 194);}
\item[(d)] We call the bounded complex {$\mathcal{K}_{X/Y}:=f^!\mathcal{O}_Y$ the \emph{relative dualizing complex}, and $\omega^o_{X/Y}:= \mathcal{H}^0( \mathcal{K}_{X/Y}[-r])$ the \emph{relative dualizing sheaf}. In particular if $Y = \text{Spec }k$, then the relative dualizing complex  is called the \emph{dualizing complex}, denoted by $\mathcal{K}_{X}$; and $\omega^o_{X}:= \mathcal{H}^0( \mathcal{K}_{X}[-r])$ is called the \emph{dualizing sheaf};}
\item[(e)]{If $f$ is a Cohen-Macaulay morphism (\cite{Ha} p. 298), i.e., $f$ is flat and all the fibers are Cohen-Macaulay, then $f^!\mathcal{O}_Y\simeq \omega^o_{X/Y}[r]$ for a quasi-coherent sheaf $\omega^o_{X/Y}$, and $\omega^o_{X/Y}$ is compatible with base change (\cite{Ha} p. 388). If moreover $f$ is a Gorenstein morphism, i.e. $f$ is flat and all the fibers are Gorenstein, then $\omega^o_{X/Y}$ is an invertible sheaf (\cite{Ha} p. 298).}
\end{itemize}

\begin{Proposition}[\cite{Vie1} Lemma 6.4]\label{pa} Let $h:X\rightarrow S$ be a projective Cohen-Macaulay morphism, $g:Y\rightarrow S$ a projective Gorenstein morphism between quasi-projective varieties.
Let $f: X \rightarrow Y$ be a  projective morphism over $S$.
Then $f^!\mathcal{O}_Y\simeq \omega_{X/Y}^o[r]$ and $\omega^o_{X/Y} \simeq \omega_{X/S}^o\otimes  (f^*\omega^o_{Y/S})^{-1}$ where $r = \dim X - \dim Y$.
\end{Proposition}

\begin{proof}
Let $n=\dim X/S$ and $m=\dim Y/S$. By (e), $\mathcal{K}_{X/S} \simeq \omega^o_{X/S}[n]$ and $\mathcal{K}_{Y/S} \simeq \omega^o_{Y/S}[m]$ where $\omega^o_{X/S}$ is a sheaf and $\omega^o_{Y/S}$ is an invertible sheaf. By (a) and (c), we have
$$\mathcal{K}_{X/S} \simeq f^!\mathcal{K}_{Y/S} \simeq f^!(\mathcal{O}_{Y} \otimes^L \omega^o_{Y/S}[m]) \simeq f^!\mathcal{O}_{Y} \otimes^L f^*\omega^o_{Y/S}[m].$$
So $f^!\mathcal{O}_Y \simeq \omega^o_{X/S} \otimes  (f^*\omega^o_{Y/S})^{-1}[r]$.
\end{proof}

\begin{Proposition}[\cite{Vie1} Corollary 6.5]\label{pf}
Let $X/S$ and $Y/S$ be two projective varieties over $S$ of the same relative dimension $r$, and $f: X/S \rightarrow Y/S$ a projective morphism over $S$. Assume that $\mathcal{K}_{X/S} \simeq \omega^o_{X/S}[r]$, $\mathcal{K}_{Y/S} \simeq \omega^o_{Y/S}[r]$, both $\omega^o_{X/S}$ and $\omega^o_{Y/S}$ are invertible sheaves, and $Rf_*\mathcal{O}_X \simeq \mathcal{O}_Y$. Then there is a natural injection $f^*\omega^o_{Y/S} \rightarrow \omega^o_{X/S}$.
\end{Proposition}

\begin{proof} The duality isomorphism gives $\text{Ext}_X^i(F,f^{!}G)\rightarrow \text{Ext}^i_Y(Rf_*F,G)$ (\cite{Ha}, p. 210).  For $i=0$, one has $$\text{Hom}_X(\mathcal{O}_X,f^!\mathcal{O}_Y)\cong\text{Hom}_Y(Rf_*\mathcal{O}_X,\mathcal{O}_Y)\cong \text{Hom}_Y(\mathcal{O}_Y,\mathcal{O}_Y).$$
By the assumption and Proposition \ref{pa}, we have $f^{!}\mathcal{O}_Y\simeq \omega_{X/S}^o\otimes(f^*\omega_{Y/S}^o)$. Therefore there is an injection $\mathcal{O}_X\rightarrow f^!\mathcal{O}_Y\simeq\omega_{X/S}^o\otimes (f^*\omega_{Y/S}^o)^{-1}$, corresponding to $1\in \text{Hom}_Y(\mathcal{O}_Y,\mathcal{O}_Y)$ by the above Proposition.
\end{proof}

\begin{Proposition}\label{pb}
Let $f: X \rightarrow S$ and $u: S' \rightarrow S$ be two projective morphisms between projective Gorenstein varieties. Considering the base change
\[\begin{CD}
X' = X \times_S S'     @>v >>          X \\
@Vf'VV       @V fVV    \\
S' @> u >>        S
\end{CD} \]
If either $u$ or $f$ is flat, then $f'^{!}\mathcal{O}_{S'}\simeq \omega_{X'/S'}^o[r]$ where $r=\dim X-\dim S$ and $\omega_{X'/S'}^o \cong v^*\omega_{X/S}^o$.
\end{Proposition}

\begin{proof}
If $u$ is flat, then the assertion follows from (b). If $f$ is flat, since both $X$ and $S$ are Gorenstein, all the fibers of $f$ are Gorenstein (\cite{Ha}  Prop. 9.6).  The assertion follows from (e).
\end{proof}

We study the behavior of the relative dualizing sheaf under a non-flat base change.

\begin{Theorem}\label{compds}
Let $f: X \rightarrow Z$ be a fibration of relative dimension $r$ and $\pi: Z' \rightarrow Z$ a generically finite surjective morphism, where $X,Z,Z'$ are projective varieties over $k$. Let $\sigma: X' \rightarrow \bar{X}':=X\times_Z Z'$ be a morphism mapping $X'$ birationally onto the component of $\bar{X}'$ dominating over $X$ (namely, the strict transform, see Definition \ref{D:Strict_T}). Fix these morphisms into the following commutative diagram
\[\begin{CD}
X' @>\sigma>> \bar{X}' = X \times_Z Z'     @>\pi_1 >>          X \\
@V hVV    @V g'VV       @V fVV    \\
Z' @>id>> Z' @> \pi >>        Z
\end{CD} \]
Assume that $f^!\mathcal{O}_{Z} \simeq \omega^o_{X/Z}[r]$, $h^!\mathcal{O}_{Z'} \simeq \omega^o_{X'/Z'}[r]$, and $\omega^o_{X/Z}, \omega^o_{X'/Z'}$ are invertible sheaves on $X$ and $X'$ respectively. Then there exists an effective $\sigma$-exceptional divisor $E$ on $X'$ such that
$$\omega^o_{X'/Z'} \leq \sigma^*\pi_1^*\omega^o_{X/Z} +E.$$
\end{Theorem}

\begin{proof} The projective morphism $f:X\rightarrow Z$ can be factored through a closed imbedding $i:X\rightarrow P$ and a smooth morphism $
p:P\rightarrow Z$, such that $f=pi$. Let $P'=P\times_ZZ'$ be the base change.  Consider the following commutative diagram
$$\xymatrix{X'\ar[ddr]_h\ar[r]^{\sigma}&\bar{X}'\ar[r]^{\pi_1}\ar[d]^j&X\ar[d]^i\\
&P'\ar[r]^{\pi_2}\ar[d]^{p'}&P\ar[d]^p\\
&Z'\ar[r]^{\pi}&Z}$$
where $g'=p'j$.

We obtain the morphisms $$\xymatrix{R\sigma_*h^{!}\mathcal{O}_{Z'}\ar[r]^{\alpha}&g'^{!}\mathcal{O}_{Z'}\ar[r]^{\beta}&L\pi^*_1f^{!}\mathcal{O}_Z},\eqno{\spadesuit}$$
where $\alpha$ is induced by trace map and $\beta$ is an isomorphism on a nonempty open subset of the strict transform of $X$ in $\bar{X}'$.

Indeed, since $h=g'\sigma$, we have $h^{!}\simeq \sigma^{!}g'^{!}$. The morphism $\alpha$ is simply the trace map $R\sigma_*\sigma^{!}\rightarrow \text{id}_{D^+(\bar{X}')}$ applied to $g'^{!}\mathcal{O}_{Z'}$.

The morphism $\beta$ is a base change map, obtained by applying $j^{-1}$ to the morphism $$\begin{array}{rcl}j_*g'^{!}L\pi^*\mathcal{O}_Z&\simeq& j_*j^{!}p'^{!}L\pi^*\mathcal{O}_Z\simeq R\mathcal{H}om_{\mathcal{O}_{P'}}(j_*\mathcal{O}_{X}',p'^{!}L\pi^*\mathcal{O}_Z)\\
&\simeq& R\mathcal{H}om_{\mathcal{O}_{P'}}(j_*\pi_1^*\mathcal{O}_X,p'^*L\pi^*\mathcal{O}_Z\otimes \omega^o_{P'/Z'}[d])\\
&\xrightarrow{\gamma}& R\mathcal{H}om_{\mathcal{O}_{P'}}(L\pi^*_2i_*\mathcal{O}_X,
p'^*L\pi^*\mathcal{O}_Z\otimes \omega^o_{P'/Z'}[d])\\
&\simeq& R\mathcal{H}om_{\mathcal{O}_{P'}}(L\pi^*_2i_*\mathcal{O}_X,L\pi_2^*(p^*\mathcal{O}_Z\otimes
\omega^o_{P/Z}[d]))\\
&\simeq& L\pi_2^*R\mathcal{H}om_{\mathcal{O}_{P}}(i_*\mathcal{O}_X,p^*\mathcal{O}_Z\otimes\omega^o_{P/Z}[d])\\
&\simeq& L\pi_2^*i_*i^{!}p^!\mathcal{O}_Z\xrightarrow{\delta}j_*L\pi_1^*i^{!}p^{!}\mathcal{O}_Z\\
&\simeq& j_*L\pi_1^*f^!\mathcal{O}_Z,
\end{array}$$
where $d=\dim P/Z$, $\gamma$ and $\delta$ are both given by the base change map $L\pi_2^*i_*\rightarrow j_*L\pi_1^*$. If $V$ is an open subvariety of $P'$ on which $\pi_2$ is flat, then $\beta|_{j^{-1}(V)}$ is an isomorphism.

By taking 0\textsuperscript{th} cohomology of $\spadesuit$, we get a homomorphism of sheaves
$$\sigma_*\omega^o_{X'/Z'} \rightarrow \mathcal{H}^0(g'^!\mathcal{O}_{Z'}[-r])\rightarrow \pi_1^*\omega^o_{X/Z}$$
and its pull-back homomorphism
$$\heartsuit: \sigma^*\sigma_*\omega^o_{X'/Z'} \rightarrow \sigma^*\pi_1^*\omega^o_{X/Z}.$$

The natural homomorphism $\lambda: \sigma^*\sigma_*\omega^o_{X'/Z'} \rightarrow \omega^o_{X'/Z'}$ is surjective outside the $\sigma$-exceptional locus, and since $\omega^o_{X'/Z'}$ is a line bundle, we can write the image of $\lambda$ as $I\otimes \omega^o_{X'/Z'}$ where $I\subset \mathcal{O}_{X'}$ is an ideal. The subvariety defined by $I$ is contained in the exceptional locus of $\sigma$. Observing that $\ker\lambda$ is torsion and thus $\text{Hom}_{X'}(\ker\lambda, \sigma^*\pi_1^*\omega^o_{X/Z}) = 0$, applying $\text{Hom}_{X'}(-, \sigma^*\pi_1^*\omega^o_{X/Z})$ to the following exact sequence
$$0\rightarrow \ker\lambda \rightarrow \sigma^*\sigma_*\omega^o_{X'/Z'} \rightarrow I\otimes \omega^o_{X'/Z'} \rightarrow 0$$
we deduce that the homomorphism $\heartsuit: \sigma^*\sigma_*\omega^o_{X'/Z'} \rightarrow \sigma^*\pi_1^*\omega^o_{X/Z}$ factors through a homomorphism $I\otimes \omega^o_{X'/Z'} \rightarrow \sigma^*\pi_1^*\omega^o_{X/Z}$. We conclude that there exists an effective $\sigma$-exceptional divisor $E$ such that $\omega^o_{X'/Z'} \leq \sigma^*\pi_1^*\omega^o_{X/Z} + E$.
\end{proof}

\subsection{Canonical sheaves}
It is known that the canonical sheaf $\omega_X$ coincides with the dualizing sheaf for a smooth projective variety $X$. For a normal quasi-projective variety $X$, we define the canonical sheaf $\omega_X:=i_*\omega_{X^{sm}}$ where $i:X^{sm} \hookrightarrow X$ is the inclusion of the smooth locus. If $X$ is a normal projective Gorenstein variety, then $\omega_X$ is an invertible sheaf coincides with the dualizing sheaf $\omega_X^o$.

Let $f: X\rightarrow Y$ be a morphism  between two projective varieties over $k$ of relative dimension $r$. We define the \emph{relative canonical sheaf} $\omega_{X/Y}$ as
\begin{itemize}
\item
the relative dualizing sheaf  $\omega^o_{X/Y}$, if $Y$ is Gorenstein and $X$ is Cohen-Macaulay, or if $f$ is a Cohen-Macaulay morphism;
\item
$\omega_X^o \otimes (f^*\omega^o_Y)^{-1}$, if $Y$ is Gorenstein and $X$ is normal.
\end{itemize}
The two definitions are compatible, and  we have $f^!\mathcal{O}_Y[-r] \cong \omega_{X/Y}$ in the first case.

\subsection{Kodaira dimension}

In \cite{Iit}, Iitaka defines $D$-dimension $\kappa(X,D)$ for a Cartier divisor $D$ on a complete normal variety $X$. If $X$ is a smooth projective variety over $k$, the Kodaira dimension $\kappa(X):=\kappa(X,\omega_X)$, where $\omega_X$ is the canonical sheaf of $X$.

\begin{Definition}[\cite{ue} Chap II, Definition 6.5,  \cite{la} Example 2.1.5] \label{D:Kod}
Let $X$ be a singular projective variety over $k$. We say $X'$  a \emph{smooth model} of $X$, if $X'$ is smooth projective over $k$ and $X'$ is birational to $X$. If $X$ has a smooth model $X'$, then the Kodaira dimension $\kappa(X)$ is defined by $\kappa(X):=\kappa(X')$.
\end{Definition}

\begin{Remark}
1) Resolution of singularities in a positive characteristic field hasn't been settled yet, so the existence of a smooth model is not clear. However, if $\dim X=1$ or $2$, smooth models of $X$ always exist over $k$.

2) If smooth models of $X$ exist, then the definition of Kodaira dimension is independent of choice of the smooth models.
\end{Remark}

The other way to define Kodaira dimension is to use dualizing sheaf (\cite{ha77} Chap III, Section 7).

\begin{Definition} \label{D:Kod1}
 Let $X$ be a projective variety over $k$ with the invertible dualizing sheaf $\omega_X^o$. Then $\kappa_1(X):=\kappa(X,\omega_X^o)$.
\end{Definition}

\begin{Remark}
1) For a fibration $f:X\rightarrow Z$ between smooth projective varieties over $k$,  the dualizing sheaf of a general fiber and the generic fiber is a line bundle.

2)  If $X$ is smooth, then the dualizing sheaf coincides with the canonical sheaf, thus $\kappa_1(X) = \kappa(X)$.
\end{Remark}

\begin{Example} Let $C'$ be a cuspidal projective curve in $\mathbb{P}^2_k$. Then $\kappa(C')=-\infty$ and $\kappa_1(C')=0$. For a  projective curve $C$ over $k$,   $\kappa_1(C)\geq \kappa(C)$.
\end{Example}

\subsection{Covering Theorem} The following theorem is needed in the sequel.
\begin{Theorem}[\cite{Iit} Theorem 10.5]\label{cl}
Let $f: X \rightarrow Y$ be a proper surjective morphism between smooth complete varieties over $k$. If $D$ is a Cartier divisor on $Y$ and $E$ an effective $f$-exceptional divisor on $X$. Then
$$\kappa(X,f^*D + E) = \kappa(Y,D).$$
\end{Theorem}

\subsection{Stable fibrations}
\begin{Definition}[Delinge and Mumford]
A flat family of nodal curves $f: X \rightarrow S$ together with sections $s_i : S \rightarrow X, i = 1,\ldots,n$ with image schemes $S_i = s_i(S)$ is called a family of $n$-pointed stable curves (or an $n$-pointed stable fibration) over $S$ of genus $g$ if

1) {$S_i$ are mutually disjoint, and  $S_i$ disjoint from the non-smooth locus $\text{Sing}(f)$;}

2) {all the geometric fibers have arithmetic genus $g$;}

3) {the sheaf $\omega_{X/S}(\sum_i S_i)$ is f-ample.}

In case $n = 0$ we simply call these stable curves (rather than stable $0$-pointed
curves).
\end{Definition}

For an $n$-pointed stable fibration $f: X \rightarrow S$ and a base change $S' \rightarrow S$, the map $X \times_S S' \rightarrow S'$ is again an $n$-pointed stable fibration. Since $f$ is a Gorenstein morphism, we know that $f^!\mathcal{O}_S[-1] \simeq \omega_{X/S}$ is an invertible sheaf, which is compatible with base change.

\subsection{Relative canonical sheaf of stable fibrations}
We  recall the following result due to Keel for a fibration of stable curves.
\begin{Theorem}[\cite{Ke} Thm. 0.4]\label{ggt2}
Let $S$ be a normal projective variety and $f: X\rightarrow S$ a stable fibration  with fibers being curves of arithmetic genus $g \geq 2$. Then
\begin{itemize}
\item[(i)]{$\omega_{X/S}$ is a semi-ample line bundle;}
\item[(ii)]{if moreover the natural map $S\rightarrow \bar{M}_g$ is finite, then $\omega_{X/S}$ is big.}
\end{itemize}
\end{Theorem}

Only (ii) is needed in the sequel.  We sketch a proof of (ii) by the argument of \cite{Ke} Proposition 4.8.
\begin{proof}[Proof of (ii)]
By taking a local basis of $f_*\omega_{X/S}$ and a proper choice of $g$ positive integers, one can define the relative wronskian section (see \cite{LT}), which is independent of the choice of the basis. Thus the wronskian section defines a non-zero global section of the sheaf $\omega^{N}_{X/S} \otimes f^*\det(f_*\omega_{X/S})^{-1}$ for some $N>0$, because the nonsmooth locus of $f$ is of codimension $\geq 2$. Then we can write that
$$\omega^N_{X/S} \equiv f^*H + Z$$
where $H \equiv \det(f_*\omega_{X/S})$ is a big line bundle on $S$ (\cite{Co} 2.2) and $Z$ is the effective divisor on $X$ defined by the wronskian section. Therefore, $\omega_{X/S}$ is big (\cite{ccp08} Lemma 2.5).
\end{proof}

\begin{Remark}
One point of the proof above is the positivity of $\det(f_*\omega_{X/S})$. The semi-positivity for fibrations with sharply $F$-pure fibers (stable curves are sharply $F$-pure) is proved by Patakfalvi (\cite{Pa} Theorem 1.5, 1.6, Corollary 1.8), it is possible to prove analogous results to Theorem \ref{ggt2} for this type of fibrations.
\end{Remark}

For a semi-stable elliptic fibration, we have the following result. We give a  proof for the readers' convenience. One can also refer to \cite{Vie2} Section 9.3 for a proof.
\begin{Theorem}\label{g1}
Let $S$ be a smooth surface and $f: S\rightarrow C$ a semi-stable elliptic fibration over a smooth curve $C$. Then $\kappa(S, \omega_{S/C}) > 0$ unless $f$ is isogenous to a product.
\end{Theorem}

\begin{proof}
First note that  $f$ is semi-stable fibration of elliptic curves, $f_*\omega_{S/C}$ is a line bundle such that $\omega_{S/C} \equiv f^*f_*\omega_{S/C}$. We have that $\omega_{S} \equiv f^*(f_*\omega_{S/C}\otimes \omega_C)$, then it follows that $c_1(S)^2 = 0$ and
\begin{equation}
\begin{split}
\chi(S, \omega_S) &= \chi(C, f_*\omega_{S/C}\otimes \omega_C) - \chi(C, R^1f_*\omega_{S/C}\otimes \omega_C)\\
& = \deg(f_*\omega_{S/C}) + \deg(\omega_C) + \chi(C, \mathcal{O}_C) - \chi(C, \omega_C) \\
& = \deg(f_*\omega_{S/C})\\
\end{split}
\end{equation}
where the second equality is from R-R formula and $R^1f_*\omega_{S/C}= \mathcal{O}_C$, and the third equality is due to $\deg(\omega_C) = 2g(C)-2$ and $\chi(C, \mathcal{O}_C) = - \chi(C, \omega_C) = 1-g(C)$.

Denote by $t$ the number of the nodes of all the fibers, which equals to the number of the singular fibers. Then $c_2(S) = t$ since Euler number of smooth fibers is 0, and of singular fibers is 1. Applying Noether's formula $12\chi(S, \omega_S) = 12\chi(S, \mathcal{O}_S)= c_1(S)^2 + c_2(S)$, we get that $12\deg(f_*\omega_{S/C}) = t \geq 0$.

If $\deg(f_*\omega_{S/C}) > 0$, then $\kappa(S, \omega_{S/C}) =1$. Otherwise we have $t = 0$, which means  all the fibers are smooth. We conclude that $f: S\rightarrow C$ is isogenous to a product by  moduli theory of curves.
\end{proof}

\section{Proof of Theorem \ref{T:Cnn-1}}

Notations are as in Theorem \ref{T:Cnn-1}. We can assume that the normalization of the geometric fiber of $f$ has genus $g \geq 1$.

\subsection{Stable reduction} \label{S:Stable_Red}
Applying de Jong's idea of alterations, we have the following commutative diagram (for a sketch of construction, see Appendix \ref{asr}).

\begin{equation}\label{Diagram}
  \begin{CD}
U  @<\rho_1 << X''= U \times_M Z' @<\rho' << X'  @> \sigma>>  \bar{X}' = X \times_Z Z'     @>\pi_1 >>          X \\
@Vh VV        @Vf'' VV   @Vf' VV      @Vg' VV       @V fVV    \\
M  @<\rho <<  Z'  @>{id} >>  Z' @>{id} >>  Z' @> \pi >>        Z
\end{CD}
\end{equation}
where
 \begin{itemize}
\item[1)] $\pi:Z' \rightarrow Z$ is an alteration (see Definition \ref{D:alter}) with $Z'$ smooth projective over $k$;

\item[2)] $\sigma:X'\rightarrow \bar{X}'$ is a birational morphism onto the strict transformation of $X$ under $Z'\rightarrow Z$ (see Definition \ref{D:Strict_T}), and $X'$ is smooth projective over $k$;

\item[3)] $\rho':X'\rightarrow X''$ is a birational morphism such that $R\rho'_*\mathcal{O}_{X'}\cong \mathcal{O}_{X''}$;

\item[4)] $h:U\rightarrow M$ is a family of stable curves with  $M$ normal varities, if the genus $g=g(\tilde{F})\geq 2$; or a family of 1-point stable curves if $g=1$. Moreover, the natural morphism $M\rightarrow \bar{M}_g$ is a finite morphism, if $g\geq 2$; or $M\rightarrow \bar{M}_{1,1}$ is finite, if $g=1$.
\end{itemize}

\begin{Claim} If $g=1$, then $\kappa(U,\omega_{U/M})=\dim M$; if $g>1$, then $\kappa(U,\omega_{U/M})=\dim U$.
\end{Claim}

\begin{proof} If $g = 1$, then $\dim M = 0$ or 1. We claim that $\kappa(U, \omega_{U/M}) = \dim M$. Indeed, it is trivial, if $\dim M =0$.  If $\dim M = 1$, then $U$ is a surface with at worst rational double singularities. Let  $\tilde{U}\rightarrow U$ be a minimal resolution. Then $\tilde{U} \rightarrow M$ is a semi-stable elliptic fibration. We get $\kappa(\tilde{U}, \omega_{\tilde{U}/M}) = 1$, by  Theorem \ref{g1}.  Thus $\kappa(U, \omega_{U/M}) = 1$.

If $g>1$,  then $\omega_{U/M}$ is big by Theorem \ref{ggt2}. Thus $\kappa(U, \omega_{U/M}) = \dim U$.
\end{proof}

\subsection{Proof of Theorem \ref{T:Cnn-1}} We can assume that $\kappa(Z)\geq 0$ and the genus $g(\tilde{F})\geq 1$. Then $\kappa(X,\omega_{X/Z})\geq 0$ by the above arguments. To prove the theorem, it suffices to prove $\kappa(X)>\kappa(Z)$, if $g(\tilde{F})\geq 2$. We prove by contradiction.

Assume now, $\kappa(F)=1$, and $\kappa(X)=\kappa(Z)$. Let $\Phi_{|\omega_X^m|}$ (resp. $\Phi_{|\omega_Z^m|}$) be the rational map induced by $|\omega_X^m|$ (resp. $|\omega_Z^m|$) for a large divisble integer $m$.
Let $X_m$ (resp. $Z_m$) be the closure of the image of $\Phi_{|\omega_X^m|}$ (resp. $\Phi_{|\omega_Z^m|}$). Since $\omega_X = \omega_{X/Z} \otimes f^*\omega_Z$, we have
$|\omega_{X/Z}^m| + |f^*\omega_Z^m| \subseteq |\omega_X^m|$.
Notice that $\kappa(X)=\kappa(Z)\geq 0 $ and $\kappa(X,\omega_{X/Z})\geq 0$.
There is a natural injection $l:f^*H^0(Z,\omega_Z^m)\rightarrow H^0(X,\omega_X^m) $ induced by tensoring a nonzero section $s\in H^0(X,\omega_{X/Z}^m)$. The injection $l$ gives a commutative diagram
$$\xymatrix{X\ar[d]_{f}\ar@{-->}[rr]^{\Phi_{|\omega_X^m|}}&&X_m\ar@{-->}[d]^p\\
Z\ar@{-->}[rr]^{\Phi_{|\omega_Z^m|}}&&Z_m}$$
where $p$ is the projection induced by $l$. The rational map $p$ is generically finite and dominant, since $\kappa(X)=\kappa(Z)$. Hence, a general
 fiber of $f$ is contracted by $\Phi_{|\omega_X^m|}$.

On the other hand, by the diagram (\ref{Diagram}), we have $\rho_1$ and $\rho'$ are dominant, $\omega_{X''/Z'}=\rho_1^*\omega_{U/M}$, and an injection $\rho'^*\omega_{X''/Z'}\hookrightarrow \omega_{X'/Z'}$ by Proposition \ref{pa}. So we have inclusions of linear systems
$$\rho'^*\rho_1^*|\omega_{U/M}^m|\subseteq\rho'^*|\rho_1^*\omega_{U/M}^m|=\rho'^*|\omega_{X''/Z'}^m|\subseteq|\rho'^*\omega_{X''/Z'}^m|
\subseteq|\omega_{X'/Z'}^m|\subseteq |\sigma^*\pi_1^*\omega_{X/Z}^m(mE)|,$$
where the last inclusion is by Theorem \ref{compds}. Let $\Psi'_m$ (resp. $\Psi_m$) be the rational map induced by $|\sigma^*\pi^*_1\omega_{X/Z}^m(mE)|$ (resp. $|\omega_{X/Z}^m|$). There is an injection $l':(\pi_1\circ \sigma)^*H^0(X,\omega_{X/Z}^m)\rightarrow H^0(X',\sigma^*\pi_1^*\omega_{X/Z}^m(mE))$, and it gives a commutative diagram $$\xymatrix{Z'\ar[d]_{\pi}&&X'\ar[ll]_{f'}\ar[d]_{\pi\sigma}\ar@{-->}[rr]^{\Psi_{m}'}&&Y_m'\ar@{-->}[d]^{p'}\\
Z&&X\ar[ll]_{f}\ar@{-->}[rr]^{\Psi_m}&&Y_m}$$
where $Y_m'$ (resp. $Y_m$) is the closure of the image of $\Psi_m'$ (resp. $\Psi_m$), $p'$ is the rational map induced by $l'$. By Theorem \ref{cl},
$\kappa(X,\omega_{X/Z})=\kappa(X',\sigma^*\pi^*_1\omega_{X/Z}+E)$, so $\dim Y_m'=\dim Y_m$. Thus the map $p'$ is generically finite and dominant.
Since $\omega_{U/M}$ is big, $\Psi_m'$ does not contract a general fiber of $f'$. So $\Psi_m$ does not contract a general fiber of $f$. Contradiction!

\begin{Remark}
If assuming $\kappa(Z) \geq 0$, the argument above proves a more optimistic inequality $\kappa(X) \geq \kappa(F) + \max\{\kappa(Z), \dim M \}$. Note that $\dim M = Var(f)$ where $Var(f)$ is the variation of $f$ (cf. \cite{Ka2}). So Conjecture $C^+_{n, n-1}$ holds (cf. \cite{Vie}).
\end{Remark}

\section{Proof of Theorem \ref{T:C21}} \label{S:C21}

\begin{Proposition}[\cite{ba01} Corollary 7.3] \label{P:fiber_integral} Let $f:X\rightarrow Y$ be a dominant morphism from an irreducible nonsingular variety $X$ of dimension at least 2 to an irreducible curve $Y$, such that the rational function field $k(Y)$ is algebraically closed in $k(X)$. Then the fiber $f^{-1}(y)$ is geometrically integral for all but a finite number of closed points $y\in Y$.
\end{Proposition}

\medskip

\begin{Definition}[\cite{ba01} Definition 7.6]
A morphism $f:S\rightarrow C$ is \emph{elliptic}  if $f$ is minimal, $f_*\mathcal{O}_S=\mathcal{O}_C$, and almost all the fibres of $f$ (i.e., except finitely many closed fibers) are nonsingular elliptic curves.

A morphism $f:S\rightarrow C$ if called \emph{quasielliptic} if $f$ is minimal, $f_*\mathcal{O}_S=\mathcal{O}_C$ and almost all the fibers of $f$ are singular integral curves of arithmetic genus one.
\end{Definition}

\medskip

Let $S$ be a smooth surface and $f:S\rightarrow C$ be an elliptic fibration or a quasi-elliptic fibration. Since $C$ is a smooth curve, we obtain a decomposition $R^1f_*\mathcal{O}_S=\mathcal{L}\oplus \mathcal{T},$ where $\mathcal{L}$ is an invertible sheaf and $\mathcal{T}$ is a torsion sheaf on $C$.

\begin{Theorem}[Canonical bundle formula, \cite{bm77} Theorem 2] \label{T:canonial} Let $S\rightarrow C$ be a relatively minimal elliptic fibration or a quasielliptic fibration from a smooth surface. Then
$$\omega_S\cong f^*(\omega_C\otimes \mathcal{L}^{-1})\otimes\mathcal{O}_S(\sum_{i=1}^ra_iP_i),$$
where

1) $m_iP_i=F_{b_i}\ \  (i=1,\ldots,r)$ are all the multiple fibers of $f$,

2) $0\leq a_i<m_i$,

3) $a_i=m_i-1$ if $F_{b_i}$ is not an exceptional fibers, and

4) $\deg (\mathcal{L}^{-1}\otimes \omega_C)=2p_a(C)-2+\chi(\mathcal{O}_S)+l(\mathcal{T})$, where $l(\mathcal{T})$ is the length of $\mathcal{T}$.
\end{Theorem}

Note that the condition $f_*\mathcal{O}_S=\mathcal{O}_C$ implies that $k(C)$ is algebraically closed in $k(S)$. So if $f:S\rightarrow C$ is a fibration, then  $f$ is a separable fibration by Proposition \ref{P:fiber_integral}.

We apply classification of surfaces by Bombieri and Mumford.

\begin{Theorem}[\cite{ba01} Corollary 10.22] If $S$ is a smooth surface with $\kappa(S)\geq 0,$ then the birational isomorphism class of $X$ contains a unique minimal model.
\end{Theorem}

\begin{Theorem}[\cite{ba01} Theorem 12.8, Proposition 13.8] \label{T:uniruled} Let $S$ be a minimal surface with $\kappa(S)=-\infty$.

1) If $h^1(S,\mathcal{O}_S)=0$ then $S$ is isomorphic to $\mathbb{P}^2$ or a Hirzebruch surface $\mathbb{F}_n:=\mathbb{P}(\mathcal{O}_{\mathbb{P}^1}\oplus\mathcal{O}_{\mathbb{P}^1}(n))\rightarrow \mathbb{P}^1$ with $n\neq 1$.

2) If $h^1(S,\mathcal{O}_S)\geq 1$ then the image $C$ of the Albanese map is a smooth curve. Moreover, there exists a rank two vector bundle $\mathcal{E}$ on $C$ such that alb$_X:X\rightarrow C$ is ismophic to $\mathbb{P}(\mathcal{E})\rightarrow C$.
\end{Theorem}

\begin{proof}[Proof of Theorem \ref{T:C21}]  We can assume that the genus $g(C)\geq 1$.

After running relative minimal model program, we can assume that $f:S\rightarrow C$ does not contain $(-1)$-curve on each fiber. If there is a horizontal $(-1)$-curve, then it will be mapped onto $C$ which has genus $g(C)\geq 1$. Contradiction! Therefore, we can assume that $f:S\rightarrow C$ with $S$ minimal (i.e. there are no $(-1)$-curves on $S$) and $g(C)\geq 1$.

There are four cases: $\kappa(S)=-\infty,0,1,2$.

1) $\kappa(S)=2$. Since $\kappa_1(F)$ and $\kappa(C)$ are at most 1, theorem holds.

2) $\kappa(S)=1$. Since $S$ is minimal, $\omega_S$ is semi-ample. Consider the Iitaka fibration on $\omega_S$, which is a morphism $g:S\rightarrow T$ with general fiber $G$. Then $G$ is an elliptic curve or a singular curve with arithmetic genus 1, and the self intersection number $G^2=0$. $$\xymatrix{S\ar[rr]^{g}\ar[d]_{f}&&T\\
C}$$
If $G$ is mapped onto $C$ by $f$, then $g(C)\leq 1$. So $g(C)=1$ and $\kappa(C)=0$. Then $\kappa(S)\geq \kappa_1(F)+\kappa(C)$.

If $G$ is mapped into a point of $C$. Notice that $G^2=0$, then $G$ is a multiple of a fiber of $f$. Then general fiber of $f$ is an elliptic curve, or a singular curve with arithmetic genus 1. Since $f$ is a flat morphism, $\kappa_1(F)=0$. So again we have $\kappa(S)\geq \kappa_1(F)+\kappa(C)$.

3) $\kappa(S)=0$. Since $S$ is minimal and $\kappa(S)=0$, we have $\omega_S\sim_{\mathbb{Q}}0$. For a general fiber $F'$ of $f$, by adjunction formula, $(\omega_S+F')|_{F'}=\omega_{F'}\sim_{\mathbb{Q}} 0$. Then $f$ is either an elliptic fibration or a quasi-elliptic fibration. By Canonical bundle formula, $$\deg f_*\omega_S^n=n(2p_a(C)-2)+n\cdot\chi(\mathcal{O}_S)+n\cdot l(\mathcal{T})+n(\sum_{i=1}^ra_i/m_i).$$
Since $\kappa(S)=0$, we have $\chi(\mathcal{O}_S)\geq 0$ (\cite{bm77}, Table above Theorem 5). Since $\kappa(S)=0$, we have $\chi(C, f_*\omega_S^n) \leq h^0(C, f_*\omega_S^n)\leq 1$. Since $\text{rank } f_*\omega_S^n = 1$ and $g(C)\geq 1$, we get  $\deg f_*\omega_S^n=0$ by R-R formula. Thus $p_a(C)=1$ and $\kappa(S)\geq \kappa_1(F)+\kappa(C)$.

4) $\kappa(S)=-\infty$. By Theorem \ref{T:uniruled}, $S$ admits a $\mathbb{P}^1$-fibration $h:S\rightarrow C'$, where $C'$ is a smooth curve. $$\xymatrix{S\ar[rr]^{h}\ar[d]_{f}&&C'\\
C}$$
Let $H$ be a general fiber of $h$, then $H$ is a smooth rational curve and the self intersection number $H^2=0$. If $H$ is horizontal of $f$, then $f(H)=C$. Contradicting to the assumption $g(C)\geq 1$. Therefore, general fiber of $f$ is $\mathbb{P}^1$. So $\kappa(S)\geq \kappa_1(F)+\kappa(C)$.
\end{proof}

\begin{Remark} \label{R:K_not_nef} We shall mention that, contrast to $\mathbb{C}$, there exist examples $f:S\rightarrow C$ from a smooth projective surface $S$ to a smooth projective curve $C$ (\cite{ra78},\cite{xie} Theorem 3.6) such that $f_*\mathcal{\omega}_{S/C}^l$ is not a nef vector bundle for any $l>0$.
\end{Remark}

By Theorem 1.3, we have

\begin{Proposition}  Let $S$ be a smooth projective surface over $k$, admitting a fibration $f:S\rightarrow C$ satisfying:

1) $C$ is a smooth projective curve with genus $g(C)\geq 2$ and

2) the general fibers of $f$ are of arithmetic genus $\geq 2$.

Then $S$ is a surface of  general type.
\end{Proposition}

\begin{Remark} \label{R:general_surf} Theorem \ref{T:Cnn-1} does not imply that $S$ is of general type.
\end{Remark}

\section{Appendix: Alterations and stable reduction} \label{asr}

In the section, we discuss the construction of the diagram (\ref{Diagram}) in Section \ref{S:Stable_Red}.
Let $f: X \rightarrow Z$ be a separable fibration between smooth projective varieties over $k$ of relative dimension 1. We assume that the normalization of the generic geometric fiber of $f$ has genus $g\geq  1$. Then  $f$ can be altered into a stable fibration due to de Jong (\cite{J}) . We sketch the construction by \cite{AO} Section 3 and 4. First we introduce some definitions.

\begin{Definition}[de Jong]\label{D:alter} A morphism of varieties $f:Y\rightarrow X$ is called a \emph{modification} if it is proper and birational.
The morphism $f$ is called an \emph{alternation} if it is proper, surjective and generically finite.
\end{Definition}

\begin{Definition}[\cite{AO} Definition 3.1] \label{D:Strict_T}
Let $f: X \rightarrow Z$ be a fibration between two varieties, $Z' \rightarrow Z$ a proper surjective morphism, where $Z'$ is a variety. The fiber product $\bar{X}': = X\times_Z Z'$ has a unique irreducible component dominant over $X$, which is the Zariski closure $X':=\overline{X \times_Z \eta} \subset X\times_Z Z'$ of the generic fiber, where $\eta$ is the generic point of $Z'$. We call $X'$ the \emph{strict transform} of $X$ under the base change $Z' \rightarrow Z$.
\end{Definition}

\emph{Step 1}. Since $f: X \rightarrow Z$ is a separable fibration, by \cite{J} Lemma 2.8, there is a finite purely inseparable extension $Z_1 \rightarrow Z$, where $Z_1$ is a projective variety, such that the normalization $X_1$ of $X \times_Z Z_1$ is smooth over the generic point of $Z_1$. Denote the fibration by $f_1:X_1\rightarrow Z_1$.

\emph{Step 2}.  There exists a modification $Z_2 \rightarrow Z_1$, such that the morphism  $f_2:X_2\rightarrow Z_2$ is flat by the Flattening Lemma (\cite{AO} Lemma 3.4), where $X_2$ is the strict transform of $Z_2$ under $X_1\rightarrow Z_1$(\cite{AO} Definition 3.1, Remark 3.2. Here need to rewrite! The definition of strict transform should be given.).

\emph{Step 3}. There exists a separable finite morphism $Z_3 \rightarrow Z_2$ such that the fibration $f_3:X_3\rightarrow Z_3$ has $l \geq 3$ distinct sections $s_i: Z_3 \rightarrow X_3, i=1,2,\ldots,l$ and the sections $s_i$ intersect every irreducible component of the geometric fiber of $f_3$ more than 2 points (\cite{AO} Lemma 4.6, 4.9), where $X_3$ is the strict transform of $Z_3$ under $X_2\rightarrow Z_2$.

\emph{Step 4}. There exists an alteration $Z_4 \rightarrow Z_3$ such that there exists a family of $l$-pointed stable curves $f_4:X_4 \rightarrow Z_4$(\cite{AO} Section 4.6 -  4.9). Here we need the assumption $l\geq 3$, so that we can apply Three Point Lemma (\cite{AO} Lemma 4.10).

\emph{Step 5}.
From $f_4:X_4 \rightarrow Z_4$, we can obtain a family of stable curves $\bar{f}_4: \bar{X}_4\rightarrow Z_4$ if general fibers of are of genus $g \geq 2$ (respectively, a family of 1-pointed stable curves if $g=1$) by ``contraction'' (\cite{AO} Section 3.7). Indeed, first deleting $l$ sections if general fibers of $f_4$ have genus $g \geq 2$ (respectively deleting $l-1$ sections if $g = 1$), then
there is a relative contraction map $X_4 \rightarrow \bar{X}_4$ over $Z_4$ by  contracting all the ``non-stable component'' of fibers (regular rational curves
containing not enough singularities and marked points).

By moduli theory of curves (cf. \cite{AO} Section 13), the moduli of the stable curves of genus $g$ if $g>1$ (1-pointed stable curves if $g=1$) with level-$m$ structure is a fine moduli space for $m\geq 3$. Thus, there is a universal family of stable curves $U_g^{(m)} \rightarrow \bar{M}_g^{(m)}$ (or 1-pointed stable curves $U_{1,1}^{(m)} \rightarrow \bar{M}_{1,1}^{(m)}$), and a natural finite surjective morphism $\bar{M}_g^{(m)} \rightarrow \bar{M}_g$ (or $\bar{M}_{1,1}^{(m)} \rightarrow \bar{M}_{1,1}$).

For simplicity, we only consider the case $g\geq 2$. The family $\bar{f}_4: \bar{X}_4\rightarrow Z_4$ induces a natural morphism $Z_4 \rightarrow \bar{M}_g$. After a finite surjective base change $Z_4' \rightarrow  Z_4$, we have a natural map $Z_4' \rightarrow  \bar{M}_g^{(m)}$, and denote its image by $M$ and the pulled back universal family by $U \rightarrow M$.

$$\xymatrix{&&U\ar[d]\ar[r]&U_g^{(m)}\ar[d]\\
&Z_4'\ar[d]\ar[r]&M\ar[r]&M_g^{(m)}\ar[d]\\
\bar{X}_4\ar[r]&Z_4\ar[rr]&&\bar{M}_g}$$

Notice that the two fiber products $\bar{X}_4 \times_{Z_4} Z_4' \rightarrow Z_4'$ and $U \times_{M} Z_4' \rightarrow Z_4'$ are two stable fibrations having isomorphic generic geometric fibers. By  \cite{AO} Lemma 3.19, there is a finite extension $Z_4'' \rightarrow Z_4'$ such that $\bar{X}_4 \times_{Z_4} Z_4'' \rightarrow Z_4''$ and $U \times_{M} Z_4'' \rightarrow Z_4''$ are isomorphic families of stable curves.

\emph{Step 6}.
Let  $Z_5 \rightarrow Z_4''$ be an alteration satisfying

1) $Z_5$ is smooth, and

2) the locus  $\Sigma \subset Z_5$, over which the  fibration $f_5: X_5 := X_4 \times_{Z_4} Z_5 \rightarrow Z_5$ is not smooth, is a simple normal crossing divisor (\cite{AO} Section 4.10).

Then
$$\bar{X}_5 := \bar{X}_4 \times_{Z_4} Z_5 \cong U \times_{M} Z_5\rightarrow Z_5$$
is a family of stable curves which is from $f_5: X_5 \rightarrow Z_5$ by contracting all the ``non-stable component'' of fibers.

The variety $X_5$ has mild singularities. That is, if $x$ is a singular point of $X_5$, then the complete local ring of $X_5$ at $x$ can be described as:
$$k[[u,v,t_1,\ldots,t_{n-1}]]/(uv - t_1^{n_1}\cdots t_r^{n_r})$$
where $n=\dim X_5$, $t_1,\ldots,t_{d-1}$ is a regular system of parameters of $f_5(x)\in Z_5$ such that $\Sigma$ coincides on a neighborhood with the zero locus of $t_1\cdots t_r$ for some $r\leq n-1$ (cf. \cite{AO} Section 4.11).

There is a resolution of singularities $\mu: X_5' \rightarrow X_5$ by blowing up (constructed in \cite{AO} Section 4.11). We can check that $R\mu_*\mathcal{O}_{X_5'} \cong \mathcal{O}_{X_5}$. For the composite morphism $\nu: X_5' \rightarrow \bar{X}_5$, we have $R\nu_*\mathcal{O}_{X_5'} \cong \mathcal{O}_{\bar{X}_5}$, because only rational curves are contracted by the morphism $X_5 \rightarrow \bar{X}_5$.

\emph{Step 7}.
In conclusion, letting $Z' \rightarrow Z$ be the composite base change $Z_5\rightarrow Z$, $X''= \bar{X}_5$ and $X' = X_5'$ introduced above, then we obtain the commutative diagram (\ref{Diagram}) in Section \ref{S:Stable_Red}.

\end{document}